\documentclass[12pt]{amsproc}
  \usepackage{latexsym} 
  \usepackage[all]{xy}
  \usepackage{amsfonts} 
  \usepackage{amsthm} 
  \usepackage{amsmath} 
  \usepackage{amssymb}
  \usepackage{pifont}  
  \usepackage{enumerate}
  \xyoption{2cell}

\newcommand{\cC}{{C}}

  \def\<{{\langle}} 
  \def\>{{\rangle}}

  \def\note#1{{}}

  \def\note#1{} 
  \def\cA{{\mathcal A}}

   \def\cK{{\mathcal K}}  
  \def\cO{{\mathcal O}}

  \def\beq{\begin{equation}} 
  \def\eeq{\end{equation}}

  \def\ot{{\otimes}}

 \def\coker{\mathrm{coker}}

 \def\WP{\mathbb{WP}}

  \newcounter{zlist} 
  \newenvironment{zlist}{\begin{list}{(\arabic{zlist})}{ 
  \usecounter{zlist}\leftmargin2.5em\labelwidth2em\labelsep0.5em 
  \topsep0.6ex%\itemsep0.3ex plus0.2ex minus0.3ex 
  \parsep0.3ex plus0.2ex minus0.1ex}}{\end{list}}

  \newcounter{blist}

  \newcounter{rlist}

\def\stac#1{\raise-.2cm\hbox{$\stackrel{\displaystyle\otimes}{\scriptscriptstyle{#1}}$}}

\def\cten#1{\raise-.2cm\hbox{$\stackrel{\displaystyle\widehat{\otimes}}
{\scriptscriptstyle{#1}}$}}

  \def\Label#1{\label{#1}\ifmmode\llap{[#1] }\else 
  \marginpar{\smash{\hbox{\tiny [#1]}}}\fi} 
  \def\Label{\label}

  \newtheorem{proposition}{Proposition}[section]
  \newtheorem{lemma}[proposition]{Lemma} 
   
  \newtheorem{theorem}[proposition]{Theorem} 

  \theoremstyle{definition}

  \theoremstyle{remark}

  \newcounter{c} 
   
  \newcommand{\etyk}[1]{\vspace{-7.4mm}$$\begin{equation}\Label{#1} 
  \addtocounter{c}{1}} 
  \renewcommand{\]}{\ifnum \value{c}=1 $$\else \end{equation}\fi} 
\def\ot{\otimes}

\def\CC{{\mathbb C}}

\def\NN{{\mathbb N}}

\def\PP{{\mathbb P}}

\def\RR{{\mathbb R}}

\def\ZZ{{\mathbb Z}}

\def\hh{{\mathfrak H}}

\newcommand{\Cc}{\mathcal{C}}

\def\m{{\sf m}}

\def\*C{{}^*\hspace*{-1pt}{\Cc}}

\def\text#1{{\rm {\rm #1}}}

 \def\l{\mathbf{l}}

 \def\k{\mathbf{k}}
  \def\e{\mathbf{e}}
 \def\p{\mathbf{p}}

 \def\1{\mathbf{1}}

\begin{document} 

\title[Quantum multidimensional teardrops]{Notes on quantum weighted projective spaces and multidimensional teardrops} 
 \author{Tomasz Brzezi\'nski}
 \address{ Department of Mathematics, Swansea University, 
  Singleton Park, \newline\indent  Swansea SA2 8PP, U.K.} 
  \email{T.Brzezinski@swansea.ac.uk}   
\author{Simon A.\ Fairfax}
\email{201102@swansea.ac.uk}

    \date{\today} 
  \begin{abstract} 
It is shown that the coordinate algebra of the  quantum $2n+1$-dimensional lens space $\mathcal{O}(L^{2n+1}_q(\prod_{i=0}^n m_i; m_0,\ldots, m_n))$ is a principal $\mathbb{Z}$-comodule algebra or the coordinate algebra of a circle principal bundle over the weighted quantum projective space $\mathbb{WP}^n_q(m_0,\ldots, m_n)$. Furthermore,  the weighted $U(1)$-action or the $\mathbb{CZ}$-coaction on the quantum odd dimensional sphere algebra $\mathcal{O}(S^{2n+1}_q)$ that defines $\mathbb{WP}^n_q(1,m_1,\ldots, m_n)$ is free or principal. Analogous results are proven for quantum real weighted projective spaces $\mathbb{RP}^{2n}_q(m_0,\ldots, m_n)$. The $K$-groups of  $\mathbb{WP}^n_q(1,\ldots, 1, m)$ and $\mathbb{RP}^{2n}_q(1,\ldots, 1,m)$ and the $K_1$-group of $L^{2n+1}_q(N; m_0,\ldots, m_n)$ are computed 
  \end{abstract} 
  \maketitle

\section{Introduction}
Quantum weighted projective spaces were introduced in \cite{BrzFai:tea} as fixed points of the weighted $U(1)$-actions on the coordinate algebra of the quantum odd dimensional sphere $\cO(S^{2n+1}_q)$. The way the circle group acts mimics the action on the classical sphere that leads to weighted projective spaces, prime examples of orbifolds. Classically these actions are not free unless all weights are mutually equal in  which case one obtains the usual complex projective spaces. In the noncommutative case the situation is much more subtle, thereby more interesting. The studies in \cite{BrzFai:tea} concentrated on the case $n=1$ and revealed that if the first of the two weights is equal to one, then a suitable finite cyclic group action on $\cO(S^{3}_q)$ produces a lens space which in turn admits the  $U(1)$-action which is free and has the quantum weighted projective space as fixed points.   Subsequently, it has been shown  in \cite{Brz:smo} that the cyclic group action that defines the lens space which is non-free classically is free in the quantum case. The combination of these observations leads one to conclude that, in a remarkable contrast to the classical situation (and in contradiction to an erroneous claim made in \cite[Theorem~3.2]{BrzFai:tea}), the $U(1)$-action on  $\cO(S^{3}_q)$ that defines the quantum weighted projective space with the first weight equal to one is free. In the classical situation this space is the teardrop orbifold with a single singular point at the north pole (i.e.\ at the point with cartesian co-ordinates $(0,0,1)$). Its coordinate algebra is determined by a polynomial with a multiple root 1, and the existence of this multiple root prevents one from differentiation and hence from constructing a unique tangent plane at the singular point. From the algebraic point of view, freeness of the action is controlled by a rational function with the denominator a polynomial in $q$, which has a root at $q=1$. The introduction of non-commutativity in the form of a parameter $q\neq 1$, splits the multiple root 1 of the defining polynomial into separate roots (powers of $q$), thus allowing for differentiation and also restores the freeness of the defining action, since the controlling rational function of $q$ is well-defined in this case. Intuitively, one can visualise this process as blowing up a singular point into a sequence of spheres.  

In further development, it has been shown in \cite{AriKaa:Pim} that it is possible to consider a suitable lens space with a free $U(1)$-action which yields the quantum weighted projective lines with no restriction on weights as fixed points. The cyclic group action on $\cO(S^{3}_q)$ that  defines these lens spaces, however, is not free.

The most recent development is the analysis of quantum weighted projective spaces for all $n$, but with a restriction on the weights $m_0,\ldots, m_n$ which classically corresponds to projective varieties isomorphic to $\CC\PP^n$  \cite{DAnLan:wei}. When this restriction is imposed one can find an explicit description of generators and D'Andrea and Landi construct principal circle bundles over quantum weighted projective spaces. Fixed points of cyclic group actions on $\cO(S^{2n+1}_q)$ serve as total spaces for these bundles. The first aim of the present note is to extend these results to quantum weighted projective spaces with no restriction on weights, and also to show that the defining circle action on  $\cO(S^{2n+1}_q)$ is free provided $m_0=1$. The second aim is to take a closer look at multidimensional quantum teardrops, i.e. $n$-dimensional quantum weighted projective spaces with all bar the last weights being equal to one. Classically such projective varieties are not isomorphic to projective spaces unless $n=1$. We describe generators of multidimensional quantum teardrops, outline their representation theory and calculate  their $K$-groups. 
 
\section{Quantum freeness and strong gradings}\label{section.strong}\setcounter{equation}{0}
In noncommutative geometry, the notion of a free quantum group action on a noncommutative space is encoded in the notion of a principal coaction of the corresponding Hopf algebra on the coordinate algebra of the quantum space; see e.g.\  \cite{BrzHaj:Che}. If the quantum group is in fact an Abelian classical group, the principality is equivalent to the notion of {\em strong grading} \cite{NasVan:gra}. 

Let $G$ be a group. A {\em $G$-graded algebra} $\cA$ decomposes into a direct sum of subspaces $\cA_g$  labelled by $g\in G$ such that $\cA_g\cA_h\subseteq A_{gh}$, for all $g,h\in G$. In case $\cA_g\cA_h= A_{gh}$, for all $g,h\in G$, $\cA$ is said to be {\em strongly $G$-graded}. We will write $|a |_G$ for the $G$-degree of $a\in \cA$. Also, $|\cA |_G$ will denote the subalgebra of $\cA$ consisting of all the invariant elements, i.e.\ all elements with degree equal to the neutral element of $G$. As explained in \cite[Section~A.I.3.2]{NasVan:gra}, $\cA$ is strongly $G$-graded if and only if, for all generators $g$ of $G$, there exist a finite number of elements $a_i,b_i\in \cA$ such that
\begin{equation}\label{res}
|a_i|_G = g^{-1}, \quad |b_i|_G = g \quad \mbox{and} \quad \sum_i a_ib_i =1.
\end{equation}

Out of a given $G$-graded algebra $\cA$, a group   epimorphism $\pi: G\to H$, and a group monomorphism $\varphi: K\to G$ one can construct the following group graded algebras. First, $\pi$ induces 
 an $H$-grading on $\cA$ by setting, for all $h\in H$,
$$
\cA_h := \bigoplus_{g\in \pi^{-1}(h)} \cA_g. 
$$
Second, $\varphi$ yields a 
$K$-graded algebra  
$$
\cA^{(K)} := \bigoplus_{k\in K} \cA_{\varphi(k)} \subseteq \cA ,
$$
i.e.\ $|a|_K =k$ if $|a|_G = \varphi(k)$.
It is shown in \cite[Section~A.I.3.1.b]{NasVan:gra} that the above $H$- and $K$-gradings are strong provided the initial $G$-grading is strong. A converse to this statement can be proven in case of Abelian groups. More precisely
\begin{lemma}[cf.\ Lemma~A.1 in \cite{Brz:Wey}]\label{lemma.tower}
Consider a short exact sequence of Abelian groups:
\begin{equation}\label{exact}
\xymatrix{ 0 \ar[r] & K \ar[r]^\varphi & G \ar[r]^\pi & H \ar[r] & 0.}
\end{equation}
A $G$-graded algebra $\cA$ is strongly graded if  and only if  the  induced $H$-grading on $\cA$  and $K$-grading on $\cA^{(K)}$ are strong. 
\end{lemma}
Note that in the case of sequence \eqref{exact}, $|\cA|_H = \cA^{(K)}$ and $|\cA|_G = |\cA^{(K)}|_K$.

In this note we deal solely with algebras graded by Abelian groups which can be fitted into an exact sequence \eqref{exact}. Specifically, $K=G=\ZZ$, $H$ is the cyclic group ${\ZZ_N}$, $\pi: \ZZ\to {\ZZ_N}$ is the canonical projection $m\mapsto m \! \mod \! N$ and $\varphi: \ZZ\to \ZZ$ is given by $m\mapsto mN$.  $\ZZ$-gradings correspond to the circle actions, while ${\ZZ_N}$-gradings correspond to actions by the cyclic group of order $N$.

\section{Quantum weighted projective and lens spaces}\setcounter{equation}{0}

The algebra $\cO(S^{2n+1}_q)$ of coordinate functions on the quantum sphere is the unital complex $*$-algebra with generators $z_0,z_1,\ldots ,z_n$ subject to the following relations:
\begin{subequations}\label{sph}
\begin{gather}
z_iz_j = qz_jz_i \quad \mbox{for $i<j$}, \qquad z_iz^*_j = qz_j^*z_i \quad \mbox{for $i\neq j$}, \label{sph1}\\
z_iz_i^* = z_i^*z_i + (q^{-2}-1)\sum_{j=i+1}^n z_jz_j^*, \qquad \sum_{j=0}^n z_jz_j^*=1, \label{sph2}
\end{gather}
\end{subequations}
where $q$ is a real number, $q\in (0,1)$; see \cite{VakSob:alg}. For any $n+1$ (coprime) positive integers $\m: = (m_0,\ldots, m_n)$, one can define the $\ZZ$-grading on $\cO(S^{2n+1}_q)$ by setting
\begin{equation}\label{grading}
|z_i|_\ZZ = m_i, \qquad |z_i^*|_\ZZ = -m_i.
\end{equation}
This grading is equivalent to the weighted $U(1)$-action on the quantum sphere or to the $\cO(U(1)) = \CC [u,u^*]$ Hopf *-algebra coaction, $z_i\mapsto z_i\ot u^{m_i}$. The degree-zero part of $\cO(S^{2n+1}_q)$ is known as the coordinate algebra of the {\em (complex) quantum $n$-dimensional weighted projective space} and is denoted by $\cO(\WP^n_q(\m))$, i.e.\
$$
\cO(\WP^n_q(\m)) := |\cO(S^{2n+1}_q)|_\ZZ.
$$
We refer to the (co)action corresponding to this grading as the {\em defining (co)action} of the quantum weighted projective space. Fix a positive integer $N$. The algebra associated to the inclusion $\ZZ \to \ZZ$, $k\to kN$ is known as the coordinate algebra of the {\em quantum lens space} and is denoted by  $\cO(L^{2n+1}_q(N;\m))$, i.e.\
$$
\cO(L^{2n+1}_q(N;\m)) := \cO(S^{2n+1}_q)^{(\ZZ)}  = |\cO(S^{2n+1}_q)|_{{\ZZ_N}};
$$
see \cite{HonSzy:len}. From this it is clear that $\cO(L^{2n+1}_q(N;\m)) \cong \cO(L^{2n+1}_q(N;\l))$ if all the $m_i$ are congruent to the $l_i$ modulo $N$. As explained in Section~\ref{section.strong}, $\cO(L^{2n+1}_q(N;\m)$ is a $\ZZ$-graded algebra with the grading induced from that of $\cO(S^{2n+1}_q)$ by the relation that $x\in \cO(L^{2n+1}_q(N;\m))$ has degree $k$ provided it has the degree $kN$ in $\cO(S^{2n+1}_q)$. The coordinate algebra of the quantum weighted projective space $\cO(\WP^n_q(\m))$ is the zero-degree part of $\cO(L^{2n+1}_q(N;\m))$.

The coordinate algebras of {\em quantum real weighted $2n$-dimensional projective spaces} $\cO(\RR\PP^{2n}_q(\m))$ are defined in a similar way. The starting point is the algebra $\cO(\Sigma^{2n+1}_q)$, which is the coordinate algebra of the $S^1$-prolongated even-dimensional quantum sphere $S_q^{2n}$. It is a $*$-algebra generated by  $z_0,\ldots, z_n$ and central unitary $w$, which in addition to relations \eqref{sph} satisfy the condition that $z_n^*=wz_n$; see \cite{BrzZie:pri}. $\cO(\Sigma^{2n+1}_q)$ is a $\ZZ$-graded algebra with (weighted) grading \eqref{grading}. Note that this determines the degree of $w$, $|w|_\ZZ = -2m_n$. The coordinate algebra of the quantum real weighted projective space $\cO(\RR\PP^{2n}_q(\m))$ is defined as the degree-zero part of $\cO(\Sigma^{2n+1}_q)$. Finally, the subalgebra of $\cO(\Sigma^{2n+1}_q)$  corresponding to the inclusion $\ZZ \to \ZZ$, $k\to kN$ is denoted by  $\cO(\Sigma^{2n+1}_q(N;\m))$.

\section{Quantum principal bundles over quantum lens and weighted projective spaces.}\setcounter{equation}{0}
To study the nature of coordinate algebras of quantum projective and lens spaces as (degree zero part of) graded algebras only partial knowledge of generators is needed. In this section the following self-adjoint elements of $\cO(\WP^n_q(\m))$ or $\cO(\RR\PP^{2n}_q(\m))$ (depending on the context)  will be of particular importance:
\begin{equation}\label{ab}
a:= \sum_{i=1}^n z_iz_i^*, \qquad   b_i := z_iz_i^*, \qquad i=0,1,\ldots, n.
\end{equation}
Noting that $z_0z^*_0 = 1 -a$ and $z^*_0z_0 = 1 - q^{-2} a$, and that $z_0a = q^2 az_0$ one easily finds that for all $N\in \NN$,
\begin{equation}\label{z0z0}
z_0^Nz_0^{*N} = \prod_{s=0}^{N-1} (1- q^{2s} a), \qquad z_0^{*N}z_0^N = \prod_{s=1}^{N} (1- q^{-2s} a).
\end{equation}

\begin{proposition}\label{prop.lens}
For all positive integers $N, m_1, \ldots ,m_n$,  and non-negative $k\in \ZZ$,
\begin{zlist}
 \item $\cO(L^{2n+1}_q(N;Nk+1,m_1,\ldots, m_n))$ is the degree-zero subalgebra of the strongly ${\ZZ_N}$-graded algebra $\cO(S^{2n+1}_q)$;
  \item $\cO(\Sigma^{2n+1}_q(N;Nk+1,m_1,\ldots, m_n))$ is the degree-zero subalgebra of the strongly ${\ZZ_N}$-graded algebra $\cO(\Sigma^{2n+1}_q)$.
  \end{zlist}
\end{proposition}
\begin{proof}
The $*$-compatible (meaning such that degree of a homogenous element is opposite to the degree of its conjugate) grading of $\cO(S^{2n+1}_q)$ or $\cO(\Sigma^{2n+1}_q)$ is given by 
$$
|z_0|_{{\ZZ_N}} = 1, \quad |z_i|_{{\ZZ_N}} = m_i \!\!\mod\! N, \qquad i=1,\ldots, n.
$$
The group ${\ZZ_N}$ is generated by 1,  hence we need to find elements $\alpha_r$, $\beta_r$ such that 
\begin{equation}\label{strong.lens}
|\alpha_r|_{{\ZZ_N}} = N-1, \quad  |\beta_r|_{{\ZZ_N}} = 1 \quad \mbox{and} \quad \sum_r \alpha_r \beta_r =1.
\end{equation}
Since $z_0^{N-1}z_0^{*N-1}$ and $z_0^*z_0$ are polynomials in $a$ with no common roots (see equations \eqref{z0z0}), by the B\'ezout lemma there exist polynomials $\alpha(a)$ and $\beta(a)$ such that
$$
\alpha(a)z_0^{N-1}z_0^{*N-1} + \beta(a)z_0^*z_0 =1.
$$
Hence 
$$
\alpha_0 = \alpha(a) z_0^{N-1}, \quad \alpha_1 = \beta(a) z_0^*, \quad \beta_0 = z_0^{*N-1}, \quad \beta_1 = z_0,
$$
satisfy all the requirements \eqref{strong.lens}.
\end{proof}

The following result, although not stated in this generality, is already contained in \cite{DAnLan:wei}.

\begin{proposition}\label{prop.weight}
For all $\m =(m_0, m_1, \dots, m_n)$, the $*$-algebras
$\cO(L^{2n+1}_q(\prod_{i}m_i;\m))$ and $\cO(\Sigma^{2n+1}_q(\prod_{i}m_i;\m))$  are strongly $\ZZ$-graded. In other words, the $2n+1$-dimensional quantum lens space $L^{2n+1}_q(\prod_{i}m_i;\m)$ is a circle principal bundle over the quantum weighted projective space $\WP^n_q(\m)$, while $\Sigma^{2n+1}_q(\prod_{i}m_i;\m)$ is such a bundle over $\RR\PP^{2n}_q(\m)$.
\end{proposition}

\begin{proof}
By \cite[Proposition~7.1]{DAnLan:wei}, for all positive integers $l_0, \ldots ,l_1$, there exist polynomials $A_0, \ldots, A_n$, $B_0, \ldots, B_n$ in $b_0, \ldots, b_n$ such that
\begin{equation}\label{weigh.res}
\sum_{i=0}^n A_iz_i^{l_i}z_i^{*l_i} = 1 \quad \mbox{and} \quad \sum_{i=0}^n B_iz_i^{*l_i}z_i^{l_i} = 1.
\end{equation}
Setting $l_i = \prod_{j\neq i}m_i$, and observing that $A_i, B_i$ are elements of  $\cO(\WP^n_q(\m))$ (resp.\  $\cO(\RR\PP^{2n}_q(\m))$), one easily finds that the $z_i^{l_i}$ and $A_iz_i^{l_i}$ are elements of $\cO(L^{2n+1}_q(\prod_{i}m_i;\m))$ (resp.\  $\cO(\Sigma^{2n+1}_q(\prod_{i}m_i;\m))$) of the $\ZZ$-degree 1, while the $z_i^{*l_i}$ and $B_iz_i^{*l_i}$   are elements of $\cO(L^{2n+1}_q(\prod_{i}m_i;\m))$ (resp.\  $\cO(\Sigma^{2n+1}_q(\prod_{i}m_i;\m))$) of the $\ZZ$-degree $-1$. Hence the first equation \eqref{weigh.res} provides one with the required resolution of identity \eqref{res} for the generator $-1\in \ZZ$, and the second of equations \eqref{weigh.res} gives such a resolution for  $1\in \ZZ$.
\end{proof}

\begin{theorem}\label{thm.free}
For all positive integers $m_1, \ldots, m_n$,
\begin{zlist}
\item The circle group action on the quantum sphere $S^{2n+1}_q$ that defines the quantum weighted projective space $\WP^n_q(1,m_1,\ldots, m_n)$ is free.
\item The circle group action on the quantum prolongated sphere $\Sigma^{2n+1}_q$ that defines the quantum weighted projective space $\RR\PP^{2n}_q(1,m_1,\ldots, m_n)$ is free.
\end{zlist}
\end{theorem}
\begin{proof}
In both cases we consider the $\ZZ$-grading  given by $|z_i|_\ZZ =m_i$, $i=1, \ldots , n$, $|z_0|_\ZZ =1$, and apply Lemma~\ref{lemma.tower} to the short exact sequence
$$
\xymatrix{ 0 \ar[r] & \ZZ \ar[r]^\varphi & \ZZ \ar[r]^\pi & {\ZZ_N} \ar[r] & 0,}
$$
where $N=\prod_i m_i$ and $\varphi$ is the inclusion map $k\mapsto Nk$. The induced ${\ZZ_N}$-gradings on $\cO(S^{2n+1}_q)$ and $\cO(\Sigma^{2n+1}_q)$ are strong by Proposition~\ref{prop.lens}, while the induced $\ZZ$-grading on $\cO(L^{2n+1}_q(N;1,m_1,\ldots, m_n))$  and  $\cO(\Sigma^{2n+1}_q(N;1,m_1,\ldots, m_n))$ are strong by Proposition~\ref{prop.weight}. Hence the  $\ZZ$-gradings of $\cO(S^{2n+1}_q)$ and  $\cO(\Sigma^{2n+1}_q)$ that define quantum weighted projective spaces are strong by Lemma~\ref{lemma.tower}, i.e.\ the corresponding actions are free in the noncommutative sense.
\end{proof}

The results of this section indicate smoothness of quantum spaces discussed here which is in a marked contrast with the classical situation. An algebra is said to be {\em homologically smooth} if it admits a resolution of finite length by its finitely generated projective bimodules; see e.g.\ \cite[Erratum]{Van:rel}.  Homological smoothness of algebras discussed in this paper can be argued thus: The enveloping algebras  of coordinate algebras of quantum unitary groups are Ore localizations of iterated skew polynomial algebras, hence they  are Noetherian and of finite global dimension (see e.g.\ \cite[Chapters~I.2  \& II.9]{BroGoo:lec}). Since quantum odd dimensional spheres are homogeneous spaces of quantum unitary groups,  the enveloping algebras  of $\cO(S_q^{2n+1})$ are also Noetherian and of finite global dimension by \cite[Corollary~6]{Kra:Hoc}. The strongness of the gradings that define $\cO(L^{2n+1}_q(N;Nk+1,m_1,\ldots, m_n))$ and $\cO(\WP^n_q(1,m_1,\ldots, m_n))$ imply that these algebras are homologically smooth; see \cite[Criterion~1]{Brz:smo}. The existence of surjective algebra homomorphisms $\cO(S_q^{2n+1})\to \cO(S_q^{2n})$ implies that also enveloping algebras of $\cO(S_q^{2n})$ are Noetherian of finite global dimension. The algebras $\cO(\Sigma^{2n+1}_q)$ are defined in \cite{BrzZie:pri} as degree-zero parts of a strong $\ZZ_2$-grading on the Laurent polynomial extension of  $\cO(S_q^{2n})$, so the quantum coordinate algebras  $\cO(\Sigma^{2n+1}_q)$, $\cO(\Sigma^{2n+1}_q(N;Nk+1,m_1,\ldots, m_n))$ and $\cO(\RR\PP^{2n}_q(1,m_1,\ldots, m_n))$ are homologically smooth.

\section{The $K_1$-groups of the quantum lens spaces}\setcounter{equation}{0}
The interpretation of $C^*$-algebras of continuous functions on quantum spheres as graph $C^*$-algebras in \cite{HonSzy:sph} allows one to find the formula for the $K_0$-group and effectively compute the $K_1$-group for quotients of quantum spheres by finite group actions; see \cite{Cri:cor}. For the ${\ZZ_N}$ action defining $L^{2n+1}_q(N;\m)$ this computation is carried out in  \cite{HonSzy:len}, provided all the $m_i$ are coprime with $N$, but the arguments used there can be easily extended to the general case, thus leading to the following

\begin{proposition}\label{prop.K}
For all $N$, $\m = (m_0, \ldots, m_n)$, where all the $m_i \in \{0,\ldots, N - 1\}$ are pairwise coprime,
\begin{equation}\label{K1}
K_1(C(L^{2n+1}_q(N;\m))) = \ZZ^{\sum_{i=0}^n \gcd(N,m_i) -n}.
\end{equation}
\end{proposition}

\begin{proof}
Recall from \cite{HonSzy:sph} that the algebra of continuous functions on the quantum odd-dimensional sphere, $C(S_q^{2n+1})$, is isomorphic to the graph $C^*$-algebra associated to the finite graph $L_{2n+1}$ defined as 
follows. $L_{2n+1}$ has $n+1$ vertices $v_0, v_1, \ldots, v_n$ and $(n+1)(n+2)/2$ edges $e_{ij}$, $i=0,\ldots ,n$, $j = i, \ldots ,n$, with $v_i$ the source and $v_j$ the range of $e_{ij}$. As explained in \cite[Lemma~2.3]{KalQui:skew}, the grading of a graph $C^*$-algebra by a finite group $G$ or, equivalently, the coaction by a finite group algebra corresponds to a labeling of the graph, i.e.\ to a  function that to each edge assigns an element of $G$. In the case of the grading of $C(S_q^{2n+1})$ that yields $C(L^{2n+1}_q(N;\m))$, the corresponding labeling is $c: e_{ij}\mapsto m_j \! \mod \! N$. Associated to this labeling is the crossed product graph $L_{2n+1}\times_c {\ZZ_N}$ with edges $(e_{ij},m)$ and vertices $(v_i,m)$, where $m=\{0,\ldots ,N-1\}$. The vertex $(v_i, m-m_i)$ is the source of $(e_{ij},m)$, while its range is $(v_j,m)$. By \cite[Theorem~4.6]{Cri:cor}, the zero degree part of $C(S_q^{2n+1})$, i.e.\ $C(L^{2n+1}_q(N;\m))$, is the full corner of the $C^*$-algebra $C(L_{2n+1}\times_c {\ZZ_N})$ associated to $L_{2n+1}\times_c {\ZZ_N}$. Hence, by \cite[Chapter~3]{RaeWil:Mor}, the $K$-groups of $C(L^{2n+1}_q(N;\m))$ are isomorphic to that of $C(L_{2n+1}\times_c {\ZZ_N})$.  In view of \cite[Theorem~3.2]{RaeSzy:Cun}, the latter can be read off the following homomorphism of Abelian groups:
$$
\Phi: \ZZ^{N (n+1)} \to \ZZ^{N (n+1)}, \quad \lambda_i^m \mapsto  \sum_{j=0}^i\lambda_j^{(m-m_j)\!\!\!\mod\! N} -  \lambda_i^m, \quad m=0,\ldots, N-1,\ i = 0,\ldots , n.
$$
So, specifically,
$$
K_0(C(L^{2n+1}_q(N;\m))) = \coker\, \Phi, \qquad K_1(C(L^{2n+1}_q(N;\m))) = \ker \Phi.
$$
We calculate the kernel of $\Phi$ by slightly amending the method of the proof of \cite[Lemma~2.1]{HonSzy:len}. Set $n_i:= \gcd(N,m_i)$. For $i=0$, the condition
\begin{equation}\label{kernel}
\sum_{j=0}^i\lambda_j^{(m-m_j)\!\!\!\mod\! N} -  \lambda_i^m = 0, \qquad m=0,\ldots, N-1,\ i = 0,\ldots , n,
\end{equation}
implies that 
$
\lambda_0^{(kn_0 +l_0)\!\!\!\mod\! N} = \lambda_0^{l_0}$, for all $l_0 = 0,\ldots, n_0 -1$.
So there are at most $n_0$ independent $\lambda_0^m$. Summing up  the equality \eqref{kernel} with $i=1$ over all the $m\in \ZZ_N$, 
yields the constraint
$
\sum_{s=0}^{n_0} \lambda_0^s = 0.
$
Furthermore, \eqref{kernel} with $i=1$ implies that
$$
\lambda_1^{kn_1 +l_1} = \lambda_1^{l_1} + f_{kn_1 +l_1}(\lambda_0^0,\ldots, \lambda_0^{n_0-1}), \qquad l_1 = 0,\ldots, n_1 -1.
$$
Summing up  the equality \eqref{kernel} for $\lambda_2^m$ over all the $m\in \ZZ_N$ will yields the constraint $
\sum_{s=0}^{n_1} \lambda_1^s = 0$. Repeating these arguments one therefore finds that
$$
K_1(C(L^{2n+1}_q(N;\m))) = \ker \Phi = \ZZ^{n_0-1} \oplus \ldots \oplus \ZZ^{n_{n-1}-1}\oplus \ZZ^{n_n}
 = \ZZ^{\sum_{i=0}^n \gcd(N,m_i) -n},
$$
as required.
\end{proof}

\section{$K$-theory  of quantum multidimensional teardrops}\setcounter{equation}{0}
\subsection{Quantum complex teardrops.}
To achieve a full description of generators of algebras 
$\cO(L^{2n+1}_q(N;\m))$ or
$\cO(\WP^n_q(\m))$
 is hard, although such a description can be given in some special cases. For example in  \cite{DAnLan:wei} the full list of generators is given for the quantum lens spaces and complex weighted projective spaces with weights of the form $m_i = \prod_{j\neq i} l_j$, where $l_0, \ldots , l_n$ are pairwise coprime integers. The following lemma  goes in a different direction.
\begin{lemma}\label{lemma.gen}
\begin{zlist}
\item As a $*$-algebra $\cO(\WP^n_q(1,\ldots, 1, m))$ is generated by the following elements:
\begin{equation}\label{bc}
b_{i,j} := z_iz_j^* \quad \mbox{and} \quad c_{\l} := z_0^{l_0} z_1^{l_1}\cdots  z_{n-1}^{l_{n-1}}z_n^*,
\end{equation}
where $i\leq j$, $i,j = 0,\ldots, n-1$,  $\l=(l_0,\ldots ,l_{n-1})\in \NN^n$ and $\sum_{i=0}^{n-1} l_i = m$.
\item As a $*$-algebra $\cO(L^{2n+1}_q(m;1,\ldots, 1, m))$ is generated by the following elements:
\begin{equation}\label{bc.lens}
z_n, \quad b_{i,j} := z_iz_j^* \quad \mbox{and} \quad \tilde{c}_{\l} := z_0^{l_0} z_1^{l_1}\cdots  z_{n-1}^{l_{n-1}},
\end{equation}
where $i\leq j$, $i,j = 0,\ldots, n-1$,  $\l=(l_0,\ldots ,l_{n-1})\in \NN^n$ and $\sum_{i=0}^{n-1} l_i = m$.
\end{zlist} 
\end{lemma}
\begin{proof}
We prove statement (1). The second statement is proven by similar arguments. 
$\cO(S_q^{2n+1})$ is spanned by all elements of the form  
\begin{equation}\label{basis}
z_0^{l_0}z_1^{l_1} \ldots z_{n-1}^{l_{n-1}}z_0^{*k_0}z_1^{*k_1} \ldots z_{n-1}^{*k_{n-1}}z_n^{*l_n}, \quad l_i, k_i\in \NN,
\end{equation} 
and their conjugates. 
Since $\cO(S_q^{2n+1})$ is a $\ZZ$-graded algebra with the $*$-compatible grading $|z_i|_\ZZ = 1$, $i=0,\ldots, n-1$, $|z_n|_\ZZ = m$, 
calculating the degree of elements \eqref{basis} gives the degree-zero condition,
$$
\sum_{i=0}^{n-1} l_{i}-\sum_{i=0}^{n-1}k_i=l_nm, \quad \mbox{hence, in particular,}\quad \sum_{i=0}^{n-1} l_{i} \geq \sum_{i=0}^{n-1}k_i,
$$
i.e.\ there are at least as many generators $z_0, \ldots, z_{n-1}$ as $z_0^*, \ldots, z_{n-1}^*$. Pairing up the degree-zero elements of the form $z_i z_j^*$, produces a factor proportional to the products of  $b_{i,j} = z_iz_j^*$, if $i\leq j$,  and ,  $b_{j,i}^*$, if $j>j$, which are in the list of generators \eqref{bc}. This leaves us to consider elements of the form
\begin{equation}\label{final.step}
z_{0}^{\tilde{l}_{0}} \ldots z_{n-1}^{\tilde{l}_{n-1}}z_n^{*l_n}.
\end{equation}
The degree-zero condition yields  $\sum_{i=0}^p \tilde{l}_{i}=ml_n$, which implies 
that terms \eqref{final.step} are proportional to the product of $l_n$ elements of the form $c_\l$ in \eqref{bc}.
\end{proof}

We refer to $\cO(\WP^n_q(1,\ldots, 1, m))$ as the coordinate algebra of the {\em quantum (complex multidimensional) teardrop}. The knowledge of generators of algebras $\cO(\WP^n_q(1,\ldots, 1, m))$ and  $\cO(L^{2n+1}_q(m;1,\ldots, 1, m))$ allows one to describe irreducible bounded representations of these algebras. Their representation theory derives from the representations of the quantum sphere. All irreducible $*$-representations of $\cO(S_q^{2n+1})$, which do not have $z_n$ in their kernels, are labelled by $\lambda\in \CC$ such that $ |\lambda|=1$ and denoted by $\pi^n_{\lambda}$. The orthonormal basis for the corresponding representation space $\hh_{\lambda}^n$ is $|k_0,k_1, \ldots, k_{n-1}\rangle$, $k_i=0,1,2,\ldots$. On this basis of $\hh_{\lambda}^n$, the (bounded) operators representing $z_i$ act as follows: 
 \begin{subequations}
\label{rep.2}
\begin{gather}
\label{rep.2.1}
 \pi_{\lambda}^n(z_n) |k_0, \ldots, k_{n-1}\rangle = \lambda q^{\sum_{i=0}^{n-1}(k_i+1)}|k_0, \ldots, k_{n-1}\rangle ,\\
 \pi_{\lambda}^n(z_l) |k_0, \ldots, k_{n-1}\rangle =(1-q^{2k_l})^{1/2}q^{\sum_{i=0}^{l-1}(k_i+1)}|k_0, \ldots, k_l -1,  \ldots, k_{n-1}\rangle , \;\;\; l<n;
\end{gather}
\end{subequations}
see e.g.\ \cite[Section~3]{HawLan:fred}. For all $s=0,\ldots ,m-1$, define $\hh_{\lambda,s}^n$ as a subspace of $\hh_{\lambda}^n$ spanned by all $|k_0,k_1, \ldots, k_{n-1}\rangle$ such that $\sum_{i=0}^{n-1}k_i \equiv s \!\mod \! m$. This defines the direct sum decomposition,
$$
\hh_\lambda^n = \oplus_{s=0}^{m-1} \hh_{\lambda, s}^n.
$$
Since the operators $\pi_\lambda^n(z_n)$, $\pi_\lambda^n(b_{i,i})$ are diagonal, $\pi_\lambda^n(b_{i,j})$ , $i<j$  does not change the sum of the $k_i$ in $|k_0,k_1, \ldots, k_{n-1}\rangle$, while $\pi_\lambda^n(\tilde{c}_{\l})$ decreases this sum by $m$,
$$
\pi_\lambda^n\mid_{\cO(L^{2n+1}_q(m;1,\ldots, 1, m))}( \hh_{\lambda, s}^n) \subseteq \hh_{\lambda, s}^n.
$$
Thus $\pi_\lambda^n$ restriced to $\cO(L^{2n+1}_q(m;1,\ldots, 1, m))$ splits into $m$ irreducible, unitarily inequivalent representations $\pi_{\lambda,s}^n$, $s=0, \ldots , m-1$, of $\cO(L^{2n+1}_q(m;1,\ldots, 1, m))$. Further restriction of the $\pi_{\lambda,s}^n$ to  $\cO(\WP^n_q(1,\ldots, 1, m))$ removes the dependence on $\lambda$ (two restrictions corresponding to different values of $\lambda$ are unitarily equivalent), thus producing a family of $m$ infinite dimensional representations $\pi^n_s$  of $\cO(\WP^n_q(1,\ldots, 1, m))$, with corresponding Hilbert spaces $\hh^n_{s}$. 

Since, for all $n$, there is a $*$-algebra epimorphism
\begin{equation}\label{stos}
\cO(S_q^{2n+1})\longrightarrow \cO(S_q^{2n-1}), \qquad z_i\longmapsto \begin{cases} z_i & \mbox{for}\quad i\neq n,\cr 0 & \mbox{otherwise},
\end{cases}
\end{equation}
the representations of $\cO(S_q^{2n+1})$, which contain $z_n$ in their kernel reduce to the representations $\pi_\lambda^{n-1}$ of  $\cO(S_q^{2n-1})$ acting on $\hh^{n-1}_\lambda$,  which do not contain $z_{n-1}$ in their kernel and are of the type described above. Those that send $z_{n-1}$ to zero reduce to representation of $\cO(S_q^{2n-3})$, etc. The restriction of the map \eqref{stos} to  $\cO(L^{2n+1}_q(m;1,\ldots, 1, m))$ has its image in $\cO(L^{2n-1}_q(m;1,\ldots, 1))$, the $*$-subalgebra of $\cO(S_q^{2n-1})$ generated by the $b_{i,j}$ and $\tilde{c}_\l$ in \eqref{bc.lens}. By the same arguments as before, the representation space  $\hh_\lambda^{n-1}$ splits into $m$ irreducible components $\hh_{\lambda,s}^{n-1}$ on which $\cO(L^{2n-1}_q(m;1,\ldots, 1))$ is represented. 

The map \eqref{stos}  restricted to $\cO(\WP^n_q(1,\ldots, 1, m))$ maps onto the coordinate algebra of the standard quantum complex projective space $\cO(\CC\PP_q^{n-1})$, generated by $b_{i,j}$, whose representation which does not have  $b_{i,n-1}$ in its kernel is simply the restriction of the $\cO(S_q^{2n-1})$-representation $\pi_{\lambda=1}^{n-1}$. The representations with $b_{i,n-1}$ in their kernel are those of $\cO(\CC\PP_q^{n-2})$, etc.

The algebras of continuous functions on quantum
 teardrops $\cC(\WP^n_q(1,\ldots, 1, m))$  are obtained as $C^*$-completions of 
 $\cO(\WP^n_q(1,\ldots, 1, m))$.
 The foregoing analysis of representations allows one to compute the $K$-theory of  
$\cC(\WP^n_q(1,\ldots, 1, m))$.
\begin{proposition}\label{prop.K.proj}
For all positive integers $m$,
$$
K_0(\cC(\WP^n_q(1,\ldots, 1, m)))= \ZZ^m \oplus \ZZ^n, \qquad K_1(\cC(\WP^n_q(1,\ldots, 1, m))) =0.
$$
\end{proposition}
\begin{proof}
The epimorphism \eqref{stos} gives rise to the short exact sequence 
\begin{equation}\label{ses1}
\xymatrix{0 \ar[r] & J\ar[r] & \cO(S_q^{2n+1})\ar[r] & \cO(S_q^{2n-1}) \ar[r] & 0,}
\end{equation}
where $J$ is the $*$-ideal generated by $z_n$. Note that $\pi_\lambda^{n}(z_n)$ is a compact  operator, hence $\pi_\lambda^n(J)$ is contained in the set of compact operators on $\hh^n$.   Since the maps in   \eqref{ses1} preserve the  $\ZZ$-degrees and the $\ZZ$-grading defining quantum weighted projective spaces is strong, sequence \eqref{ses1} yields the exact sequence
\begin{equation}\label{ses2}
\xymatrix{0 \ar[r] & \bar{J} \ar[r] & \cO(\WP^n_q(1,\ldots, 1, m))\ar[r] & \cO(\CC\PP_q^{n-1}) \ar[r] & 0,}
\end{equation}
where $\bar{J} = J\cap \cO(\WP^n_q(1,\ldots, 1, m))$ is generated by the $c_\l$ and $z_nz_n^* = 1 -\sum_{i=0}^{n-1} b_i$, where $b_i := b_{i,j}$ in agreement with definitions \eqref{ab}. Since, for all $x\in J$, $\pi_\lambda(x)$ is a compact operator,  $\pi_s( \bar{J})\subseteq \cK$, the ideal of compact operators on $\hh_s$. Consequently, the topological closure of $\bar{J}$ is contained in the direct sum of $m$-copies of  $\cK$. Let 
$$
c_i := z_i^mz_n^*, \qquad i=0,1\ldots , n-1.
$$
The operators $\pi_s(c_ic_i^*)$ are diagonal operators, which, by inspection of \eqref{rep.2} can be seen to distinguish between the elements of the basis of $\hh^n_s$. More precisely, let
$$
\pi_s(c_ic_i^*) |k_0, \ldots, k_{n-1}\rangle = \gamma_\k ^i|k_0, \ldots, k_{n-1}\rangle.
$$
If $i \neq  j$, then $\gamma_\k ^i\neq \gamma_{\k} ^j$, and given two $n$-tuples of natural numbers $\k$, $\k'$, if $k_i\neq k'_i$, then  $\gamma_\k ^i\neq \gamma_{\k'} ^i$. This implies that (the completion of) $\pi_s( \bar{J})$ contains all orthogonal projections to one-dimensional subspaces of $\hh^n_s$ spanned by the orthonormal basis vectors $|k_0, \ldots, k_{n-1}\rangle$, $\sum_ik_i \equiv s \!\mod\! m$. Finally, $\pi_s(c_\l)$ are step-by-one  operators on $\hh_s^n$ with non-zero weights. Hence $\pi_s(\bar{J}) = \cK$ and,  consequently, the completion of $\bar{J}$ is equal to $\cK^{\oplus m}$. Therefore, the exact sequence \eqref{ses2} gives rise the the following exact sequence of $C^*$-algebras
\begin{equation}\label{ses3}
\xymatrix{0 \ar[r] & \cK^{\oplus m} \ar[r] & \cC(\WP^n_q(1,\ldots, 1, m))\ar[r] & \cC(\CC\PP_q^{n-1}) \ar[r] & 0.}
\end{equation}
Since $K_0(\cK)=\ZZ$, $K_1(\cK)=0$ and $K_0(\cC(\CC\PP_q^{n-1})) =\ZZ^n$, $K_1(\cC(\CC\PP_q^{n-1})) =0$ (see e.g.\ \cite[Section~4]{HonSzy:sph}), the K-theory six-term exact sequence derived from \eqref{ses3} yields the $K$-groups of quantum teardrops as stated.
\end{proof}

The knowledge of generators and representations allows one also to construct 1-sum\-mable Fredholm modules over $\cO(\WP^n_q(1,\ldots, 1, m))$. Following \cite{DAnLan:bou}, consider  representations of the odd-dimensional quantum sphere algebras  $\cO(S_q^{2n+1})$ defined as follows. For a fixed $n$,  representations $\bar\pi_k^n$ are labelled by $k=0, \ldots, n$, and are defined on the space $\hh^n = \ell^2(\NN^n)$ with the orthonormal basis $|\p\rangle$, where $\p=(p_1, \ldots, p_n)\in \NN^n$. For all $0\leq i <k$ let
$\e_i^k$ denote the sequence of length $n$ consisting of $i$ zeros, followed by $k-i$ units, followed by $n-k$ zeros. Then 
$$
\bar\pi_k^n(z_i) |\p\rangle = 
\begin{cases}  
q^{p_i +i}  \begin{cases}  
\sqrt{1-q^{2(p_{i+1}-p_i})} |\p+\e_i^k\rangle, & i < k\cr
 |\p\rangle, & i =k
\end{cases} , 
& p_1\leq \ldots \leq p_k, \; p_n <\ldots < p_{k+1} \cr
 & \cr 
0, & \mbox{otherwise}
\end{cases},
$$
with the convention that $p_0=0$. Representations forming a Fredholm module are defined, for all $x\in \cO(\WP^n_q(1,\ldots, 1, m))$, by 
$$
\pi_+^n(x) = \sum_{\mbox{$k$ even}} \bar\pi^n_k(x), \qquad \pi_-^n(x) = \sum_{\mbox{$k$ odd}} \bar\pi^n_k(x).
$$
It is shown in \cite[Proposition~3]{DAnLan:bou} that $\pi_+^n(b_{i,j}) - \pi_-^n(b_{i,j})$ are trace class operators. Since the generators $c_\l$ contain $z_n^*$, 
$$
\pi_+^n(c_\l) - \pi_-^n(c_\l) = (-1)^n \bar\pi^n_n(c_\l).
$$
The matrix coefficients of $\pi^n_n(c_\l)$ are bounded by $q^{p_n}$, and hence the trace of $\pi^n_n(c_\l)$ is bounded by
$$
\sum_{p_1 \leq p_2\leq \ldots \leq p_n}q^{p_n} = \sum_{r=0}^\infty\binom{r+n-1}{n-1}q^r,
$$
which is convergent, for all  $q\in (0,1)$. Therefore, $\pi_+^n(x) - \pi_-^n(x)$ is a trace class operator for all $x\in \cO(\WP^n_q(1,\ldots, 1, m))$. This yields a 1-summable Fredholm module of the standard form $(\hh^n\oplus\hh^n , \pi^n_{+} \oplus \pi^n_- ,F, \gamma)$,  
where
$$
 F = \begin{pmatrix} 0 & I \cr I & 0\end{pmatrix}, \qquad \gamma = \begin{pmatrix} I & 0 \cr 0 & -I\end{pmatrix}.
$$

\subsection{Quantum real teardrops.}
The generators of $*$-algebras $\cO(\RR\PP^{2n}_q(1,\ldots, 1, m))$ and $\cO(\Sigma^{2n+1}_q(m;1,\ldots, 1, m))$ can be described in a way similar to the complex teardrop case.
\begin{lemma}\label{lemma.gen.real} Let $m$ be a positive integer.
\begin{zlist}
\item As a  $*$-algebra $\cO(\RR\PP^{2n}_q(1,\ldots, 1, m))$  is generated by the following elements:
\begin{equation}\label{bc.real}
b_{i,j} := z_iz_j^* ,\quad c_{\l} := z_0^{l_0} z_1^{l_1}\cdots  z_{n-1}^{l_{n-1}}z_n w \quad \mbox{and} \quad d_{\p} := z_0^{k_0} z_1^{k_1}\cdots  z_{n-1}^{k_{n-1}}w 
\end{equation}
where $i\leq j$, $i,j = 0,\ldots, n-1$,  $\l=(l_0,\ldots ,l_{n-1}), \p=(p_0,\ldots ,p_{n-1})\in \NN^n$, $\sum_{i=0}^{n-1} l_i = m$ and $\sum_{i=0}^{n-1} p_i = 2m$.
\item  As a $*$-algebra $\cO(\Sigma^{2n+1}_q(m;1,\ldots, 1, m))$ is generated by the following elements:
\begin{equation}\label{bc.lens.real}
z_n, \quad w,\quad  b_{i,j} := z_iz_j^* \quad \mbox{and} \quad \tilde{c}_{\l} := z_0^{l_0} z_1^{l_1}\cdots  z_{n-1}^{l_{n-1}},
\end{equation}
where $i\leq j$, $i,j = 0,\ldots, n-1$,  $\l=(l_0,\ldots ,l_{n-1})\in \NN^n$ and $\sum_{i=0}^{n-1} l_i = m$.
\end{zlist} 
\end{lemma}
\begin{proof}
The proof is similar to that of Lemma~\ref{lemma.gen}, and, as in that proof we only justify statement (1). 
$\cO(\Sigma_q^{2n+1})$ is spanned by all elements of the form  
\begin{equation}\label{basis.sei}
z_0^{l_0}z_1^{l_1} \ldots z_{n-1}^{l_{n-1}}z_0^{*k_0}z_1^{*k_1} \ldots z_{n-1}^{*k_{n-1}}z_{n}^r w^s, \quad l_i, k_i, r\in \NN,\; s\in \ZZ;
\end{equation} 
see \cite{BrzZie:pri}. To derive generators \eqref{bc.real} we focus on those elements in the basis \eqref{basis.sei}, for which $2s\geq r$. Since $(z_n^rw^s)^* = z_n^r w^{r-s}$,   the elements such that $2s< r$ are simply $*$-conjugates of those with $2s\geq r$. Following the same arguments as in the proof of Lemma~\ref{lemma.gen}, the elements \eqref{basis.sei} are in  $\cO(\RR\PP^{2n}_q(1,\ldots, 1, m))$ provided
$$
\sum_{i=0}^{n-1} l_{i}-\sum_{i=0}^{n-1}k_i=m(2s-r) \geq 0.
$$
Thus we can group the $z_i$ with the $z_j^*$ to form the $b_{i,j}$ and be left to consider only those in \eqref{basis.sei} for which $k_0=\ldots = k_{n-1} =0$. The degree-zero condition $\sum_{i=0}^{n-1} l_{i}=m(2s-r)$ splits into two cases. If $s\geq r$, then
$$
z_0^{l_0}\ldots z_{n-1}^{l_{n-1}}z_{n}^r w^s \propto \prod_{i=1}^r c_{\l^i} \prod_{i=1}^{s-r} d_{\p^i}\, .
$$
If $s<r$, then set $k:= r-s$, so that $2s - r = r-2k$, and then
$$
z_0^{l_0}\ldots z_{n-1}^{l_{n-1}}z_{n}^r w^s \propto (z_n^2w)^k \prod_{i=1}^{r-2k} c_{\l^i} = \left(1-\sum_{j=0}^{n-1}b_{i,i}\right)^k \ \prod_{i=1}^{r-2k} c_{\l^i}\, .
$$
Therefore,  $\cO(\RR\PP^{2n}_q(1,\ldots, 1, m))$ is generated by elements listed in \eqref{bc.real}.
\end{proof}

In analogy to to $\cO(\WP^n_q(1,\ldots, 1, m))$,  $\cO(\RR\PP^{2n}_q(1,\ldots, 1, m))$ is referred to as the coordinate algebra of the {\em quantum real (multidimensional) teardrop}. Representations of algebras  $\cO(\RR\PP^{2n}_q(1,\ldots, 1, m))$ and $\cO(\Sigma^{2n+1}_q(m;1,\ldots, 1, m))$ arise from representations of $\cO (\Sigma^{2n+1}_q)$; see \cite[Section~5.2]{BrzZie:pri}. These split into two classes. If $z_n$ is not mapped to zero, the representation space $\hh_{\lambda,\pm}^n$ of $\pi_{\lambda,\pm}^n$, $|\lambda|=1$,   has an orthonormal basis: $|k_0,k_1, \ldots, k_{n-1}\rangle$, $k_i=0,1,2,\ldots$ and 
 \begin{subequations}
\label{rep.1.real}
\begin{gather}
 \pi_{\lambda,\pm}^n(z_n) |k_0, \ldots, k_{n-1}\rangle = \pm \lambda  q^{\sum_{i=0}^{n-1}(k_i+1)}|k_0, \ldots, k_{n-1}\rangle ,\\
 \pi_{\lambda,\pm}^n(z_l) |k_0, \ldots, k_{n-1}\rangle =(1-q^{2k_l})^{1/2} q^{\sum_{i=0}^{l-1}(k_i+1)}|k_0, \ldots, k_l -1,  \ldots, k_{n-1}\rangle , \quad l<n, \\
  \pi_{\lambda,\pm}^n(w) |k_0, \ldots, k_{n-1}\rangle =  \lambda^{-2} |k_0, \ldots, k_{n-1}\rangle .
\end{gather}
\end{subequations}
Restriction of these representations to $\cO(\RR\PP^{2n}_q(1,\ldots, 1, m))$ and $\cO(\Sigma^{2n+1}_q(m;1,\ldots, 1, m))$ splits them into $m$ irreducible components, as in the complex case. 
If $z_n$ is represented by the zero operator, then representations of $\cO(\Sigma^{2n+1}_q)$ become the same as those of $\cO(S^{2n-1}_q)$ tensored with the circle representation of $w$. The restrictions of the map  defined in $\cO(\Sigma^{2n+1}_q)$ by the formula  \eqref{stos} to  $\cO(\RR\PP^{2n}_q(1,\ldots, 1, m))$ and $\cO(\Sigma^{2n+1}_q(m;1,\ldots, 1, m))$  give rise to epimorphisms
$$
 \cO(\Sigma^{2n+1}_q(m;1,\ldots, 1, m)) \longrightarrow \cO(L^{2n-1}_q(2m; 1,\ldots ,1))\ot \cO(S^1), 
 $$
 and 
 \begin{equation}\label{ontolens}
 \cO(\RR\PP^{2n}_q(1,\ldots, 1, m))  \longrightarrow \cO(L^{2n-1}_q(2m; 1,\ldots ,1)),
\end{equation}
 and hence the representations in this case are derived from those of the lens and circle algebras.
 
 By the same arguments as in the proof of Proposition~\ref{prop.K.proj}, the epimorphism \eqref{ontolens} yields a short exact sequence of $C^*$-algebras
\begin{equation}\label{ses4}
\xymatrix{0 \ar[r] & \cK^{\oplus m} \ar[r] & \cC(\RR\PP^{2n}_q(1,\ldots, 1, m))\ar[r] & \cC(L_q^{2n-1}(2m;1,\ldots ,1)) \ar[r] & 0,}
\end{equation}
which allows one to calculate the $K$-groups of $\cC(\RR\PP^{2n}_q(1,\ldots, 1, m))$. By \cite[Proposition~2.2]{HonSzy:len}, $K_1(\cC(L_q^{2n-1}(2m;1,\ldots ,1))) = \ZZ$. The $K_0$-group of $\cC(L_q^{2n-1}(2m;1,\ldots ,1))$ has been calculated recently in \cite{AriBra:Gys}:
$$
 K_0(\cC(L_q^{2n-1}(2m;1,\ldots ,1))) = \ZZ\oplus \bigoplus_{i=1}^{n-1} \ZZ_{r_i},
 $$
 where the positive integers $r_i$ are defined as follows. Consider the $n\times n$-matrix $A$ with only non-zero entries 
 $$
 a_{ij} = (-1)^{i-j+1}\binom{2m}{i-j}, \qquad 0<i-j\leq \min(2m,n-1).
 $$ 
 Let $d_i$, $i=1,\ldots , n-1$ denote the greatest common divisors of non-zero minors of $A$ of size $i$. Then $r_1 := d_1$ and, for $i\neq 1$, $r_i := d_i/d_{i-1}$. In view of this analysis of $K$-groups of  quantum lens spaces, the sequence \eqref{ses4} yields the following six-term exact sequence of $K$-groups
 $$
 \xymatrix{0\ar[r] &  K_1(\cC(\RR\PP^{2n}_q(1,,\ldots, 1, m)))\ar[r] & \ZZ\ar[d]^\delta \\
 \ZZ\oplus \bigoplus_{i=1}^{n-1} \ZZ_{r_i} \ar[u] & \ar[l] K_0(\cC(\RR\PP^{2n}_q(1,\ldots, 1, m))) & \ar[l] \ZZ^m,}
 $$
 with the connecting map $\delta: i\mapsto (2i,\ldots ,2i)$. This form of $\delta$ is supported by the fact that in the ideal of $\cO(\RR\PP^{2n}_q(1,\ldots, 1, m))$ obtained as the restriction of the $*$-ideal of  $\cO (\Sigma^{2n+1}_q)$ generated by $z_n$, the only step operators are derived from the $d_\p$, and since $\sum_{i=0}^{n-1} k_i =2m$, the representation of $d_\p$ on each of the $\hh^{n-1}_{\lambda, s}$ has the form of a step-by-two operator. The map $\delta$ is injective and $\coker(\delta) = \ZZ^{m-1}\oplus \ZZ_2$, therefore,
 $$
K_1(\cC(\RR\PP^{2n}_q(1,\ldots, 1, m))) =0,
$$
and there is a short exact sequence
\begin{equation}\label{ses.real}
\xymatrix{0 \ar[r] & \ZZ^{m-1} \oplus \ZZ_2\ar[r] & K_0(\cC(\RR\PP^{2n}_q(1,\ldots, 1, m)))\ar[r] & \ZZ \oplus  \bigoplus_{i=1}^{n-1} \ZZ_{r_i} \ar[r] & 0}.
\end{equation} 
Note that $d_{n-1} = 2^{n-1}m^{n-1}$ and since $d_{n-1} = r_1r_2\ldots r_{n-1}$ at least one of the $r_i$ is even if $n>1$.   
As a consequence, the sequence \eqref{ses.real} as it stands  does not determine uniquely the $K_0$-group. More precisely,
 $K_0(\cC(\RR\PP^{2n}_q(1,\ldots, 1, m)))$ might be equal to 
$$
 \ZZ^{m} \oplus \ZZ_2\oplus \bigoplus_{i=1}^{n-1} \ZZ_{r_i} \quad \mbox{or} \quad  
\ZZ^{m} \oplus \ZZ_{2r_k}\oplus  \bigoplus_{\stackrel{i=1}{ i\neq k}}^{n-1} \ZZ_{r_i},
$$
for any even $r_k$. For example, 
for $n=2$, $r_1 =2m$, hence, 
$$
K_0(\cC(\RR\PP^{4}_q(1, 1, m)))= \ZZ^{m} \oplus \ZZ_2 \oplus \ZZ_{2m} \quad  \mbox{or} \quad  
K_0(\cC(\RR\PP^{4}_q(1, 1, m)))=  \ZZ^{m} \oplus \ZZ_{4m} ,
$$
while  for $n=3$, $r_1 =m $ and $r_2 =4m$, hence there are three possible forms of the group $K_0(\cC(\RR\PP^{6}_q(1, 1,1, m)))$ when $m$ is even
$$
\ZZ^{m} \oplus \ZZ_2 \oplus \ZZ_{m} \oplus \ZZ_{4m}, \qquad  \ZZ^{m} \oplus \ZZ_{2m} \oplus \ZZ_{4m}  \quad \mbox{or} \quad \ZZ^{m} \oplus \ZZ_{m} \oplus \ZZ_{8m}  ,
$$
with the  first two mutually isomorphic  if $m$ is odd. For a general $n$ and $m=1$, $r_{n-1} =2^{n-1}$ and $r_i =1$, for $i<n-1$, so there are two options for the $K_0$-group one of which is $K_0(\cC(\RR\PP^{2n}_q(1,\ldots, 1)))= \ZZ \oplus  \ZZ_{2^{n}}$,  which is in agreement with \cite[Section~5]{HonSzy:sph}.

For $n=1$, the third group in  \eqref{ses.real} has no torsion part, hence the sequence \eqref{ses.real}  splits and so
$K_0(\cC(\RR\PP^{2}_q(1, m)))= \ZZ^m \oplus \ZZ_2$ in agreement with \cite[Section~4.3]{Brz:Sei}.

 \section*{Acknowledgments}
 We would like to thank Tyrone Crisp for an interesting conversation about corners of graph $C^*$-algebras and in particular for drawing reference \cite{Cri:cor} to our attention. We would also like to thank Ulrich Kr\"ahmer for bringing reference \cite{BroGoo:lec} to our attention.

\end{document}